\documentclass{article}

\usepackage{makeidx}
\usepackage{latexsym}
\usepackage{amsmath}
\usepackage{amssymb}
\usepackage{amscd}
\usepackage{amsthm}
\usepackage{multicol}
\usepackage{color}
\usepackage[all,cmtip]{xy}
\usepackage{graphicx}
\usepackage{wrapfig}

\usepackage[applemac]{inputenc} %

\newtheorem{theorem}{Theorem}[section]
\newtheorem{teorema}[theorem]{Theorem}
\newtheorem{proposicion}[theorem]{Proposition}
\newtheorem{lema}[theorem]{Lemma}
\newtheorem{corolario}[theorem]{Corollary}

\newtheorem{lemma}[theorem]{Lemma}
\newtheorem{corollary}[theorem]{Corollary}

\theoremstyle{definition}
\newtheorem*{definition}{Definition} 
\newtheorem{example}[theorem]{Example}
\newtheorem{remark}[theorem]{Remark}
\newtheorem*{notation}{Notation}

\def\proof{\noindent \textit{Proof: }}

\def\d{\mathrm{d}}

\def\CC{\mathbb{C}}

\def\RR{\mathbb{R}}

\def\g{\mathfrak{g}}

\def\Man{{\bf Man}}

\def\Ad{{\rm Ad}\,}

\def\qed{\hfill $\square$}
\def\D{\mathrm{d}}

\begin{document}

\title{Natural operations on differential forms}

\author{J. Navarro \thanks{Corresponding author. Email address: {\it navarrogarmendia@unex.es} \newline 
Department of Mathematics, Universidad de Extremadura, Avda. Elvas s/n, Badajoz, Spain. \newline The first author has been partially supported by Junta de Extremadura and FEDER funds. }  
 \and  J. B. Sancho }

\maketitle

\begin{abstract}
We prove that the only natural operations between differential forms are those 
obtained using linear combinations, the exterior product  and the exterior differential. Our result generalises work by Palais \cite{Palais} and Freed-Hopkins \cite{FreedHopkins}. 

As an application, we also deduce a theorem, originally due to Kol\'{a}\v{r} (\cite{Kolar}), that determines those natural differential forms that can be asso\-ciated to a connection on a principal bundle.
\end{abstract}

\bigskip


\tableofcontents

\section*{Introduction}

Let $\, \omega_1\dots,\omega_k\,$ be differential forms on a smooth manifold, of positive degree. In this paper we determine those differential forms that can be associated in a {\it natural} way to the given forms $\, \omega_1, \ldots , \omega_k$. Loosely speaking, our result says that the only natural operations between differential forms are those obtained using linear combinations, the exterior product  and the exterior differential.

To be more precise, let us fix positive integers $\, p_1,\dots,p_k\,$ and an integer $\, q\geq 0$. Let us suppose that, for each manifold $\,X\,$ and each collection $\,\omega_1\dots,\omega_k\,$ of differential forms on $\,X\,$ of degree $\,p_1,\dots,p_k$, we have defined a $q$-form $\,P(\omega_1,\dots,\omega_k)\,$ on $\,X$. Let us also assume that the assignment $\,(\omega_1\dots,\omega_k)\longmapsto P(\omega_1,\dots,\omega_k)\,$ is compatible with inverse images: for any smooth map $\,\tau\colon\bar X\to X\,$ it holds 
$$\tau^*\left[ P(\omega_1,\dots,\omega_k)\right]\,=\, P(\tau^*\omega_1,\dots,\tau^*\omega_k)\ .$$

Then, our main result (Theorem \ref{3.3}) proves that there exists a unique polynomial $\,\mathbf{P}(x_1,y_1,\dots,x_k,y_k)$, homogeneous of degree $q$ in anti-commutative variables of degree $\,\text{deg}\, x_i=p_i$, $\text{deg}\, y_i=p_i+1$, such that:
$$P(\omega_1,\dots,\omega_k)\,=\, \mathbf{P}(\omega_1,\D\omega_1,\dots,\omega_k,\D\omega) \ , $$
for any smooth manifold $\,X\,$ and any collection $\,\omega_1\dots,\omega_k\,$ of differential forms on $\,X\,$ of degrees $\,p_1,\dots,p_k$. In this formula, the product of  variables on the polynomial $\,\mathbf{P}\,$ has been replaced by the exterior product of forms. 

The case $\,p_1=\dots=p_k=1\,$ has been recently obtained by Freed-Hopkins (\cite{FreedHopkins}, Thm. 8.1) using a different language and quite different methods to those applied here.

When the given collection reduces to a single form $\,\omega\,$ of degree $\,p$, and we take $\,q=p+1$, it follows that the assignment $\,P\,$ is a constant multiple of the exterior differential: $\,P(\omega)=\lambda \, \D\omega$. This is a classical statement due to Palais (\cite{Palais}, Thm. 10.5), who assumed $\, P\,$ to be linear, and to Kol\'{a}\v{r}-Michor-Slov\'{a}k (\cite{KMS}, Prop. 25.4) in the general case.

We have included two versions of our main result: a first one, Theorem \ref{2.3}, whose proof relies on general methods developed in \cite{KMS}, and a second one, Theorem \ref{3.3}, which corresponds to the statement announced above.
The exposition of this paper will be self-contained, except for the Peetre-Slov\'{a}k theorem on differential operators and some basic facts about the invariant theory of the general linear group. 

As an application of our results, we prove in the Appendix a version of a beautiful theorem originally due to Kol\'{a}\v{r} (\cite{Kolar}), that determines those differential forms that one can construct in a natural way from a connection on a principal bundle (see also \cite{FreedHopkins}, Thm. 7.20).

\section{Preliminaries}

\subsection{Invariant theory of the general linear group}

Let $\,V\,$ be an $\mathbb{R}$-vector space of finite dimension $\,n$, and let $\,Gl_n\,$ be the Lie group of its $\mathbb{R}$-linear automorphisms.

As $\,Gl_n\,$ is linearly semisimple, the following holds:

\begin{proposicion}\label{Semisimple}
Let $\,E\,$ and $\,F\,$ be linear representations of $\,Gl_n$, and let $\,E'\subset E\,$ be a sub-representation. 
Any equivariant linear map $\,E' \to F\,$ is the restriction of an equivariant linear map $\,E \to F$.
\end{proposicion}

On the other hand, the Main Theorem of the invariant theory for the general linear group states that the only $\,Gl_n$-equivariant linear maps $\,\otimes^p V \longrightarrow \otimes^p V\,$ are the linear combinations of permutations of indices (e. g., \cite{KMS} Sect. 24). In this paper, we will only use the following consequence:

\begin{proposicion}\label{CorolarioMain}  The only $\,Gl_n$-equivariant linear maps $\, \otimes^p V \longrightarrow \Lambda^p V\,$ are the multiples of the skew-symmetrisation operator.
\end{proposicion}

Combining the previous two propositions, it follows:

\begin{corolario}\label{0.3}  Let $\,E\subseteq\,\otimes^pV\,$ be a $\,Gl_n$-sub-representation. The only $\,Gl_n$-equi\-va\-riant linear maps $\,\otimes^pV\supseteq E \longrightarrow \Lambda^p V\,$ are the multiples of skew-symme\-trisation operator.
\end{corolario}

\begin{remark}\label{0.4} We will make use of the following properties of the skew-sy\-mme\-trisa\-tion operator $\,h$:\smallskip

-- $h(T\otimes T')=h(T)\wedge h(T')\,$ for any covariant tensors $\,T,T'$,\smallskip

-- $h(\omega)=q!\,\omega$ for any $q$-form $\,\omega$,\smallskip

-- $h(\nabla\omega)=q!\,\D\omega$ for any differential $q$-form $\omega$ on $\,\mathbb{R}^n$, where $\,\nabla\,$ denotes the standard flat connection.
\end{remark}

\subsection{Differential operators on forms}

Let  $\, E\to X\,$ be a bundle (or a submersion) on a smooth manifold $\,X$. Let us denote by $\,\mathcal{E}\,$ the sheaf of sections of that bundle: $\,\mathcal{E}\,$ is a contravariant functor, defined over the category of open sets of $\,X\,$ and inclusions between them, that assigns, to each open set $\,U\subseteq X$, the set $\,\mathcal{E}(U)\,$ of smooth sections defined over $\,U$, and to each inclusion between open sets $\,V\hookrightarrow U$, the restriction map $\,\mathcal{E}(U)\to\mathcal{E}(V)$.

Let us denote by $\,J^rE\to X\,$ the bundle of $r$-jets of sections of $\,E\to X$.\medskip

\begin{definition}
The \textbf{bundle of $\infty$-jets} is the inverse limit $$J^\infty E\, :=\,\varprojlim J^rE\ ,$$ endowed with the initial topology of the projections $\, \pi_r\colon J^\infty E\to J^rE$. 
\end{definition}

Let $\,Y\,$ be a smooth manifold. A continuous map $\,\varphi\colon J^\infty E\to Y\,$ is said \textbf{smooth} if, for any $\infty$-jet $\,j^\infty_xs\in J^\infty E\,$ there exist a natural number $\,r\,$ and a smooth map $\,\varphi_r\colon J^rE\supseteq U\to Y$, defined on a neighbourhood $\,U\,$ of $\,j^r_xs$, such that $\,\varphi=\varphi_r\circ\pi_r$ in the neighbourhood $\,\pi_r^{-1}(U)\,$ of $\,j^\infty_xs$.\medskip

\begin{definition} Let $\,E\to X\,$ and  $\,F\to X\,$ be two bundles. A \textbf{differential operator} $\,\tilde P\colon E\leadsto F\,$ is a smooth morphism of bundles $\,\tilde P\colon J^\infty E\to F$. A differential operator $\,\tilde P\,$ is said of \textbf{order} $\leq r$ if there exists a morphism of bundles $\,\tilde P_r\colon J^rE\to F\,$ such that $\,\tilde P=\tilde P_r\circ\pi_r$.
\end{definition}

Let $\,\mathcal{E}\,$ and $\,\mathcal{F}\,$ be the sheaves of sections of two bundles $\,E\to X\,$ and $\,F\to X$. A differential operator $\,\tilde P\colon E\leadsto F$ can be understood as a morphism of sheaves (that is, a morphism of functors):
 $$P\colon\mathcal{E}\,\longrightarrow\,\mathcal{F}\quad,\quad P(s)(x):=\,\tilde P(j^\infty_xs)\ .$$

Conversely, let us see that any morphism of sheaves, satisfying certain re\-gu\-larity condition, is a differential operator.\medskip
 
If $\,T\,$ is a smooth manifold, let us denote $\,X_T := X\times T$. Any open set $\,U \subset X_T\,$ can be thought as a family of open sets $\,U_t \subset X$, where $\,U_t=U\cap(X\times t)$. 

A family of sections $\,\{ \, s_t \colon U_t \to E \, \}_{t \in T}\,$ defines a map: 
$$ s \colon U \to E \quad , \quad s (t,x) := s_t (x) \  $$ and the family $\,\{ s_t \}_{t \in T}\,$ is said {\it smooth} (with respect to the parameters $\,t \in T$) precisely when the map $\,s \colon U \to E\,$ is  smooth.
 
\begin{definition} With the previous notations, a morphism of sheaves $\,P \colon \mathcal{E} \to \mathcal{F}\,$ is \textbf{regular} if, for any smooth manifold $\,T\,$ and any smooth family of sections $\,\{ s_t \colon U_t \to E \}_{t \in T}$, the family $\,\{ P(s_t) \colon U_t \to F \}_{t\in T}\,$ is also smooth.
\end{definition}

\begin{theorem}[Peetre-Slovak]\label{PeetreSlovak} There exists a bijection
$$\mathrm{Diff}(E,F)\,=\,\mathrm{Hom}_{\mathrm{reg}}(\mathcal{E},\mathcal{F})\quad ,\quad \tilde P\mapsto P\ ,$$
where $\,\mathrm{Diff}(E,F)\,$ stands for the set of differential operators and $\,\mathrm{Hom}_{\mathrm{reg}}(\mathcal{E},\mathcal{F})\,$ for that of regular morphisms of sheaves.
\end{theorem}

This theorem was first proved by J. Peetre, in the case of $\mathbb{R}$-linear differential operators. Later on, J. Slov\'{a}k (\cite{Peetre}) established a general result that includes the above statement as a particular case (see \cite{PeetreRevisited}).\medskip

\begin{notation}
If $\,X\,$ is a smooth manifold, let us write $\,\Omega^p_X\,$ for the sheaf of differential $p$-forms on $\,X$; that is, for the sheaf of sections of the bundle $\,\Lambda^pT^*X\to X$.
\end{notation}

\begin{lema}\label{1.2} Any differential operator $\,\tilde P\colon\Lambda^pT^*\mathbb{R}^n\leadsto\Lambda^qT^*\mathbb{R}^n$, that is invariant under translations and homotheties, has finite order.
\end{lema}

The condition of a differential operator $\,\tilde P\,$ being invariant by a diffeomorphism $\,\tau\colon\mathbb{R}^n\to\mathbb{R}^n\,$ means that the following square is commutative
$$\begin{CD}
J^\infty\Lambda^pT^*\mathbb{R}^n @>{\tau^*}>> J^\infty\Lambda^pT^*\mathbb{R}^n\\
@V{\tilde P}VV @VV{\tilde P}V \\
\Lambda^qT^*\mathbb{R}^n @>{\tau^*}>> \Lambda^qT^*\mathbb{R}^n
\end{CD}$$

In terms of the corresponding morphism of sheaves $\,P\colon\Omega^p_{\mathbb{R}^n}\to\Omega^q_{\mathbb{R}^n}$, this amounts to saying $\,P(\tau^*\omega)=\tau^*P(\omega)$.\medskip

\noindent{\it Proof of \ref{1.2}:} As the operator $\,\tilde P\,$ is invariant under translations, it is determined by its restriction $\,\tilde P\colon J^\infty_0(\Lambda^pT^*\mathbb{R}^n) \longrightarrow \Lambda^qT^*_0\mathbb{R}^n$ to the fibres over the origin $\,0\in\mathbb{R}^n$. The smoothness of this map at the zero $\infty$-jet $\,j^\infty_00\in J^\infty_0(\Lambda^pT^*\mathbb{R}^n)\,$ implies that there exist a natural number $\,r\,$ and a smooth map $\,\tilde P_r\colon J^r_0(\Lambda^pT^*\mathbb{R}^n)\supseteq U \longrightarrow \Lambda^qT^*_0\mathbb{R}^n$, defined on an open set $\,U\,$ of the zero $r$-jet, such that $\,\tilde P=\tilde P_r\circ\pi_r\,$ over $\,\pi_r^{-1}(U)$.

Now, a simple computation in local coordinates shows that, for any fixed jet $\,j^\infty_0\omega\in J^\infty_0\Lambda^pT^*\mathbb{R}^n$, there exists an homothety $\,\tau_\lambda\colon\mathbb{R}^n\to\mathbb{R}^n$, $\,\tau_\lambda(x)=\lambda x$, with $\,\lambda\,$ sufficiently small,  such that $\,\tau_\lambda^*(j^r_0\omega)\,$ belongs to $\,U$, that is, $\,\tau_\lambda^*(j^\infty_0\omega)\in \pi_r^{-1}(U)$. Therefore, 
\begin{align*}
\tilde P(j^\infty_0\omega)\, &=\,\tilde P(\tau_{\lambda^{-1}}^*\tau_\lambda^*(j^\infty_0\omega))\,=\,\tau_{\lambda^{-1}}^*\left(\tilde P(\tau_\lambda^*(j^\infty_0\omega))\right)\,=\,\tau_{\lambda^{-1}}^*\left(\tilde P_r\left(\pi_r\tau_\lambda^*\left(j^\infty_0\omega\right)\right)\right) \\
&=\, \tau_{\lambda^{-1}}^*\left(\tilde P_r\left( \tau_\lambda^*\pi_r\left(j^\infty_0\omega\right)\right)\right)\,=\,\tau_{\lambda^{-1}}^*\left(\tilde P_r\left(\tau_\lambda^*\left(j^r_0\omega\right)\right)\right)
=\,\tilde Q_r(j^r_x\omega)\ .
\end{align*}
where $\,\tilde Q_r :=\tau_{\lambda^{-1}}^*\circ\tilde P_r\circ\tau_\lambda^*$.

The same formula $\,\tilde P=\tilde Q\circ\pi_r\,$ holds for any point in the neighbourhood of $\,j^\infty_0\omega\,$ formed by jets $\,j^\infty_0\bar\omega\,$ such that $\,\tau_\lambda^*(j^\infty_0\bar\omega)\in \pi_r^{-1}(U)$. 
Hence, $\,\tilde P\,$ factors through the projection $\,\pi_r$.

\hfill$\square$

\medskip
\begin{notation} 
For the rest of the section, let us fix a natural number $\, n$.  Let $\,\textbf{Man}_n\,$ denote the category of smooth manifolds of dimension $\,n\,$ and local diffeomorphisms between them.
\end{notation}

Let $\, \Omega^p\,$ denote the contravariant functor over $\,\textbf{Man}_n\,$ that assigns, to each smooth $n$-manifold $\,X$, the set $\,\Omega^p(X)\,$ of differential $p$-forms on $\,X$, and, to each local diffeomorphism $\,\tau\colon X\to Y$, the map $\,\tau^*\colon\Omega^p(Y)\to\Omega^p(X)\,$ (inverse image of differential forms).

\begin{definition} A \textbf{natural} morphism $\,P\colon\Omega^p\to\Omega^q\,$ is a morphism of functors over $\,\textbf{Man}_n$. 

In other words, a natural morphism $\,P\colon\Omega^p\to\Omega^q\,$  amounts to giving, for each $n$-manifold $\,X$, a map $\,P\colon\Omega^p(X)\to\Omega^q(X)$ such that for each local diffeomorphism $\,\tau\colon X\to Y$, it holds 
$$P(\tau^*\omega)\,=\,\tau^*(P(\omega))\qquad \forall \, \omega\in\Omega^p(Y)\ .$$
\end{definition}

\medskip
Let $\,X\,$ be an $n$-manifold. The sheaf of $p$-forms $\,\Omega_X^p\,$ over $\,X\,$ is the restriction of $\,\Omega^p\,$ to the subcategory of $\,\textbf{Man}_n\,$ formed by the open sets of $\,X\,$ and the inclusions between them. 
A natural morphism $\,P\colon\Omega^p\to\Omega^q\,$ defines, by restriction to such subcategory, a morphism $\,P\colon\Omega_X^p\to\Omega_X^q\,$ of sheaves over $X$.

\begin{definition} A natural morphism $\,P\colon\Omega^p\to\Omega^q\,$ is said \textbf{regular} if its restriction $\,P\colon\Omega^p_X\to\Omega^q_X\,$ to any $n$-manifold $\,X\,$ is a regular morphism of sheaves over $\,X$.
\end{definition}

Let $\,P\colon\Omega^p\to\Omega^q\,$ be a regular and natural morphism. As any $n$-manifold $\,X\,$ is locally diffeomorphic to $\,\mathbb{R}^n$, the natural morphism $\,P\colon\Omega^p\to\Omega^q\,$ is determined by its restriction to the manifold $\,\mathbb{R}^n$, that is, by the regular morphism of sheaves $\,P\colon\Omega^p_{\mathbb{R}^n}\to\Omega^q_{\mathbb{R}^n}$. By the Peetre-Slov\'{a}k theorem, such a morphism is in fact defined by a differential operator $\,\tilde P\colon J^\infty\Lambda^pT^*\mathbb{R}^n\longrightarrow \Lambda^qT^*\mathbb{R}^n$. By naturalness, this differential operator is invariant under the action of any diffeomorphism $\,\mathbb{R}^n\to\mathbb{R}^n$, and hence, due to Lemma \ref{1.2}, has finite order. 

Summing up,

\begin{proposicion}\label{1.3} Any regular and natural morphism $\,P\colon\Omega^p\to\Omega^q\,$ is a differential operator of finite order.\end{proposicion}

To be more precise, the restriction of $\,P\,$ to each $n$-manifold $\,X\,$ is a morphism of sheaves $\,P\colon\Omega^p_X\to\Omega^q_X\,$ corresponding to a differential operator $\,\tilde P\colon \Lambda^pT^*X\leadsto\Lambda^qT^*X\,$ of finite order. This order, that is common for all manifolds, will be called the \textbf{order} of the morphism $\,P\colon\Omega^p\to\Omega^q$.

\begin{definition} For each natural number $\,r$, let us write $\,G^r\,$ for the Lie group of $r$-jets $\,j^r_0\tau\,$ at the origin $\,0\in\mathbb{R}^n\,$ of local diffeomorphisms satisfying $\,\tau(0)=0$.
\end{definition}

The group $\,G^{r+1}\,$ acts on the fibre $\,J^r_0(\Lambda^pT^*\mathbb{R}^n)$,
$$(j^{r+1}_0\tau)\cdot (j^r_0\omega):=\, j^r_0\left(\tau_*\omega\right)\ .$$

\begin{proposicion}\label{1.4}
There exists an injective map:
$$\left\{\begin{aligned}\mathrm{Regular\ and \ natural \ morphisms \ }\\
P\colon\Omega^p\longrightarrow\Omega^q\ \mathrm{\ of \ order \ }r\quad\end{aligned}\right\}\,\subseteq\,
 \left\{\begin{aligned}\mathrm{Smooth \ }G^{r+1}\text{--}\,\mathrm{equivariant \ maps \ }\\ \tilde P\colon J^r_0(\Lambda^pT^*\mathbb{R}^n)\longrightarrow \Lambda^qT^*_0\mathbb{R}^n\qquad\end{aligned}\right\}$$
\end{proposicion}

The relation between $\,P\,$ and $\,\tilde P\,$ is determined by the equality:
$$P(\omega)_{x=0}\,=\,\tilde P(j^r_0\omega)\qquad\forall\, \omega\in\Omega^p(\mathbb{R}^n)\ .$$

\proof Any regular and natural morphism $\,P\colon\Omega^p\to\Omega^q$ of order $\,r\,$ is determined by its restriction to $\mathbb{R}^n$, that is a regular morphism of sheaves  $\,P\colon\Omega^p_{\mathbb{R}^n}\to\Omega^q_{\mathbb{R}^n}\,$ invariant under diffeomorphisms. As it is regular, it corresponds to a differential operator $\,\tilde P\colon J^r(\Lambda^pT^*\mathbb{R}^n)\longrightarrow \Lambda^qT^*\mathbb{R}^n$, that, by naturalness, is invariant under diffeomorphisms. Due to this last property, $\tilde P$ is determined by its restriction to the fibres over the origin, $\,\tilde P\colon J^r_0(\Lambda^pT^*\mathbb{R}^n)\longrightarrow \Lambda^qT^*_0\mathbb{R}^n$, which, in turn, is a smooth,  $G^{r+1}$-equivariant map. 

\qed

With more generality, in this paper we will consider the functors $\,\Omega^{p_1}\oplus\cdots\oplus\Omega^{p_k}\,$ over $\,\textbf{Man}_n$. The definition of naturalness, as well as Propositions \ref{1.3} and \ref{1.4}, all trivially extend to morphisms $\,\Omega^{p_1}\oplus\cdots\oplus\Omega^{p_k}\longrightarrow\Omega^q$. In particular,

\begin{proposicion}\label{1.5}
There exists an injective map
$$\begin{CD}
\left\{\begin{aligned}\mathrm{Regular \ and \ natural \ morphisms \ of \ order \ }r\\
P\colon\Omega^{p_1}\oplus\cdots\oplus\Omega^{p_k}\,\longrightarrow\,\Omega^q\qquad\qquad\end{aligned}\right\}\\
|\bigcap^{\phantom{\bigcap}}
_{\phantom{\bigcap}} \\
 \left\{\begin{aligned} &\mathrm{Smooth\ }G^{r+1}\text{--}\,\mathrm{equivariant \ maps\ }\ \\ \tilde P\colon J^r_0(\Lambda^{p_1}&T^*\mathbb{R}^n)\times\cdots\times J^r_0(\Lambda^{p_k}T^*\mathbb{R}^n)\,\longrightarrow\, \Lambda^qT^*_0\mathbb{R}^n\end{aligned}\right\}\end{CD}$$
\end{proposicion}


The relation between $\,P\,$ and $\,\tilde P\,$ is determined by the equality
$$P(\omega_1,\dots,\omega_k)_{x=0}\,=\,\tilde P(j^r_0\omega_1,\dots,j^r_0\omega_k)$$
for all $\,(\omega_1,\dots,\omega_k)\in\Omega^{p_1}(\mathbb{R}^n)\times\dots\times \Omega^{p_k}(\mathbb{R}^n)$.\medskip

\begin{remark}
Later on, we will compute those smooth, $G^{r+1}$-equi\-va\-riant maps, and it will result that the inclusions of \ref{1.4} and \ref{1.5} are in fact equalities (see \cite{KMS} for more general results of this kind).
\end{remark}

\section{Main result: first version}

In this section we still work in category $\,\textbf{Man}_n\,$ of $n$-dimensional manifolds and local diffeomorphisms between them.

Let us introduce the following notations for the exterior and symmetric powers: $\,\Lambda^p:=\Lambda^p(T^*_0\mathbb{R}^n)\,$ and $\, S^p:=S^p(T^*_0\mathbb{R}^n)$. We will also use that $\,S^p\,$ is canonically isomorphic to the vector space of homogeneous polynomials of degree $\,p\,$ over $\,T^*_0\mathbb{R}^n$.

There exists a diffeomorphism:
$$\begin{CD}
J^r_0(\Lambda^pT^*\mathbb{R}^n) @= (\Lambda^p)\oplus(S^1\otimes \Lambda^p)\oplus\cdots\oplus(S^r\otimes \Lambda^p)\\
j^r_0\omega & \longmapsto & (\omega^0,\omega^1,\dots\dots ,\omega^r)\qquad\qquad 
\end{CD}$$
where $\,\omega^s\,$ is the homogeneous component of degree $\,s\,$ in the Taylor expansion (on cartesian coordinates) of the $p$-form $\,\omega$. If we write $\,\nabla\,$ for the standard flat connection on $\,\mathbb{R}^n$, then $\,\omega^s=(\nabla^s\omega)_{x=0}$.

This diffeomorphism is not invariant under arbitrary changes of coordinates, but it is so under linear changes of coordinates. In other words, this diffeomorphism is equivariant with respect to the action of the linear group $\,Gl_n \subseteq G^{r+1}$, although it is not equivariant with respect to the whole group $\,G^{r+1}$. 

Therefore, Proposition \ref{1.5} implies the following:

\begin{lema}\label{2.1} There exists an injective map:
$$\begin{CD}
\left\{\begin{aligned}\mathrm{Regular\ and \ natural \ morphisms \ of \ order \ }r\\
P\colon\Omega^{p_1}\oplus\cdots\oplus\Omega^{p_k}\,\longrightarrow\,\Omega^q\qquad\qquad\end{aligned}\right\}\\
|\bigcap^{\phantom{\bigcap}}
_{\phantom{\bigcap}} \\
 \left\{\begin{aligned} &\mathrm{Smooth\ }Gl_n\text{--}\,\mathrm{equivariant \ maps \ }\quad\\ \tilde P\colon\prod_{i=1,\dots,k} &\left[(S^0\otimes \Lambda^{p_i})\oplus\cdots\oplus(S^r\otimes \Lambda^{p_i})\right]\,\longrightarrow\, \Lambda^q\end{aligned}\right\}\end{CD}$$
 \end{lema}
 
The relation between $\,P\,$ and $\,\tilde P\,$ is determined by the equality
\begin{align*}
P(\omega_1,\dots,\omega_k)_0\, &=\,\tilde P(\dots,\left[\omega_i^0,\dots,\omega_i^r\right],\dots) \\
&=\, \tilde P(\dots,\left[(\nabla^0\omega_i)_0,\dots,(\nabla^r\omega_i)_0\right],\dots)\ .
\end{align*} 

\medskip

Again, this inclusion will be proved to be an equality, once we compute the $Gl_n$-equivariant maps. The first step in this computation is to prove that these maps satisfy certain homogeneity condition.

\begin{lema}\label{2.2} Any smooth, $\,Gl_n$-equivariant map
$$\begin{CD}\prod_{i=1}^k\left[(S^0\otimes \Lambda^{p_i})\oplus\cdots\oplus(S^r\otimes \Lambda^{p_i})\right] @>{\tilde P}>> \Lambda^q
\end{CD}$$
satisfies the following homogeneity condition
\begin{equation}\label{CondicionHomog}
\tilde P\left(\dots,(\lambda^{p_i+0}\omega_i^0,\dots,\lambda^{p_i+r}\omega_i^r),\dots\right)\,=\,\lambda^q\tilde P\left(\dots,(\omega_i^0,\dots,\omega_i^r),\dots\right)
\end{equation}
for any real number $\, \lambda\neq 0$.
\end{lema}

\proof Let $\,\tau_\lambda\colon\mathbb{R}^n\to\mathbb{R}^n$, $\tau_\lambda(x)=\lambda x$, be the homothety of ratio $\,\lambda\neq 0$. On the one hand, for any tensor $\,\omega_i^j\in S^j\otimes\Lambda^{p_i}$, covariant of order $\,p_i+j$, it holds $\,\tau_\lambda^*\omega_i^j=\lambda^{p_i+j}\omega_i^j$. 

On the other, as the map $\, P\,$ is $Gl_n$-equivariant, it holds $\,\tilde P\circ\tau_\lambda^*=\tau_\lambda^*\circ\tilde P$. Consequently:
\begin{align*}
\tilde P & \left(\dots,(\lambda^{p_i+0}\omega_i^0,\dots,\lambda^{p_i+r}\omega_i^r),\dots\right)\, =\,
\tilde P\left(\dots,(\tau_\lambda^*\omega_i^0,\dots,\tau_\lambda^*\omega_i^r),\dots\right) \\
& =\,\tau_\lambda^*\tilde P\left(\dots,(\omega_i^0,\dots,\omega_i^r),\dots\right)  \,=\,\lambda^q\tilde P\left(\dots,(\omega_i^0,\dots,\omega_i^r),\dots\right)\ .
\end{align*}

\qed

\medskip
A smooth map between vector spaces satisfying a homogeneity condition has necessarily to be a polynomial, in virtue of the following elementary result  (e. g. \cite{KMS}, Thm. 24.1):

\medskip
\noindent {\bf Homogeneous Function Theorem:} {\it Let $E_1,\dots, E_k$ be finite dimensional $\mathbb{R}$-vector spaces.

Let $\, f \, \colon  \prod E_i \to \mathbb{R}$ be a smooth function such that there exist positive real numbers $a_i > 0$, and $ w \in \mathbb{R}$ satisfying:
\begin{equation}\label{CondicionHomogeneidadLemma}
 f ( \lambda^{a_1} e_1 , \ldots , \lambda^{a_k} e_k) = \lambda^w \, f(e_1 , \ldots , e_k)
\end{equation}
 for any positive real number $\lambda >0$ and any
$(e_1 , \ldots , e_k) \in \prod E_i$.

Then $f$ is a sum of monomials of 
degree $(d_1,\dots,d_k)$ in the variables $e_1,\dots,e_k$ satisfying the relation
\begin{equation}\label{CondicionMonomios}
 a_1 d_1 + \cdots + a_k d_k = w \ .
\end{equation}
 
If there are no natural numbers  $d_1,\dots,d_k \in \mathbb{N} \cup \{ 0 \}$ satisfying this equation, then $f$ is the zero map. }
\medskip

In other words, for any finite dimensional vector space $W$, there exists an $\mathbb{R}$-linear isomorphism:
$$\begin{CD}
\left[ \text{Smooth maps }\, f \colon  \prod E_i \to W\,\text{ satisfying }  (\ref{CondicionHomogeneidadLemma})\right]  \\
@| \\
\bigoplus \limits _{(d_1 , \ldots , d_k)}\text{Hom}_{\mathbb{R}}(S^{d_1} E_1 \otimes \ldots 
\otimes S^{d_k} E_k,\, W)
\end{CD}$$
where $(d_1, \ldots , d_k)$ runs over the non-negative integers solutions of equation (\ref{CondicionMonomios}). 

A smooth map $\,f\colon\prod E_i\to W$, satisfying (\ref{CondicionHomogeneidadLemma}), and the corresponding linear map $\,\oplus f^{d_1\dots d_k}\in \bigoplus_{(d_1 , \ldots , d_k)}\text{Hom}_{\mathbb{R}}(S^{d_1} E_1 \otimes \ldots 
\otimes S^{d_k} E_k,\, W)$, are related by the equality
$$f(e_1,\dots,e_k)\,=\,\sum_{(d_1 , \ldots , d_k)}f^{d_1\dots d_k}\left((e_1\otimes\overset{d_1}{\dots}\otimes e_1)  \otimes\cdots\cdots\otimes (e_k\otimes\overset{d_k}{\dots}\otimes e_k) \right)\ .$$

\medskip

\begin{definition} Let us fix a finite sequence of positive integers $\,(p_1,\dots,p_k)$. Let us denote $\,\mathbb{R}\{ u_1,\dots,u_k\}\,$ the {\it anti-commutative} algebra of polynomials with real coefficients in the variables $\,u_1,\dots, u_k$, where each variable $\,u_i\,$ is assigned degree $\,p_i$. The anti-commutative character of this algebra is expressed by the relations
$$u_iu_j\,=\, (-1)^{p_ip_j}u_ju_i\ .$$

The degree of a monomial $\,u_1^{a_1}\dots u_k^{a_k}\,$ is defined as $\,\sum a_ip_i$. A polynomial $\,\mathbf{P}(u_1,\dots, u_k)\in \mathbb{R}\{ u_1,\dots,u_k\}\,$  is said homogeneous of degree $\,q\,$ if it is a linear combination of monomials of degree $\,q$.
\end{definition}

Let $\,\mathbf{P}(x_1,\dots, x_k)\in \mathbb{R}\{ u_1,\dots,u_k\}\,$ be a homogeneous polynomial of degree $\,q$, and let $\,\omega_1,\dots,\omega_k\,$ differential forms of degree $\,p_1,\dots,p_k\,$ on a smooth manifold $\,X$. Then $\,\mathbf{P}(\omega_1,\dots,\omega_k)$, where the product of variables is replaced by the exterior product of forms, is a differential form of degree $\,q\,$ on $\,X$.

\begin{teorema}\label{2.3} Let $\,p_1,\dots,p_k\,$ be positive integers,  and let $\,\mathbb{R}\{ u_1,v_1,\dots,u_k,v_k\}\,$ be the anti-commutative algebra of polynomials on the variables $\,u_1,v_1,\dots,u_k,v_k$, of degree $\,\mathrm{deg}\, u_i=p_i$, $\,\mathrm{deg}\, v_i=p_i+1$.

Any regular and natural morphism over $\mathbf{Man}_n$
$$ P\colon \Omega^{p_1} \oplus \ldots \oplus \Omega^{p_k} \  \longrightarrow \ \Omega^q\qquad (0\leq q\leq n)  $$
can be written as
$$P(\omega_1,\dots,\omega_r)\,=\, \mathbf{P}(\omega_1,\mathrm{d}\omega_1,\dots,\omega_k,\mathrm{d}\omega_k)$$
for a unique homogeneous polynomial $\,\mathbf{P}(u_1,v_1,\dots,u_k,v_k)\in\mathbb{R}\{ u_1,v_1,\dots,u_k,v_k\}\,$ of degree $\,q$.
\end{teorema}

\proof Via Lemma \ref{2.1}, a regular natural morphism $\,P\colon \Omega^{p_1} \oplus \ldots \oplus \Omega^{p_k}\longrightarrow\Omega^q$, of order $r$, corresponds to a $Gl_n$-equivariant smooth map
$$\tilde P\colon \prod_{i=1,\dots,k}\left[(S^0\otimes \Lambda^{p_i})\oplus\cdots\oplus(S^r\otimes \Lambda^{p_i})\right]\,\longrightarrow\, \Lambda^q\ ,$$
satisfying the homogeneity condition (\ref{CondicionHomog}).

Due to the Homogeneous Function Theorem,  such an $\,\tilde P\,$ has to be a polynomial; that is to say, it is a sum $\,\oplus_{\{d_{i,j}\}}\, \tilde P^{\{d_{i,j}\}}\,$ of $Gl_n$-equivariant linear maps:
$$ \tilde P^{\{ d_{i,j}\}}: \bigotimes_{i=1,\dots,k}\left[S^{d_{i,0}} (S^0\otimes\Lambda^{p_i}) \otimes\cdots\otimes S^{d_{i,r}} (S^r\otimes\Lambda^{p_i})\right]\, \longrightarrow \, \Lambda^{q} $$ where each sequence $\,\{ d_{i,j}\}\,$ of non-negative integers is a solution to the equation:
\begin{equation}\label{Soluciones}
\left[p_1d_{1,0}+\cdots+(p_1+r)d_{1,r}\right]+  \cdots + \left[p_kd_{k,0}+\cdots+(p_k+r)d_{k,r}\right] \, = \, q \ .
\end{equation}

Observe that this condition implies that $\,\tilde P^{\{ d_{i,j}\}}\,$ is in fact defined on a vector subspace of $\,\otimes^qT_0^*\mathbb{R}^n$.
Then, by Proposition \ref{0.3}, the linear map $\,\tilde P^{\{ d_{i,j}\}}\,$ is a multiple of the skew-symmetrisation operator $\,h$, 
$$\tilde P^{\{ d_{i,j}\}}\,=\, \lambda^{\{ d_{i,j}\}}\cdot h\ .$$

The skew-symmetrisation of two symmetric indices vanishes, so that we may assume, from now on, that
$$
d_{i,2} =d_{i,3} =\ldots =d_{i,r}= 0 \quad , \quad \mbox{ for all } i = 1 , \ldots , k.$$

That is to say, we only consider solutions $\,\{ d_{1,0},d_{1,1},\dots,d_{k,0},d_{k,1}\}\,$ to the equation:
\begin{equation}\label{SolucionesII}
\left[p_1d_{1,0}+(p_1+1)d_{1,1}\right]+  \cdots + \left[p_kd_{k,0}+(p_k+1)d_{k,1}\right] \, = \, q \ .
\end{equation}

Bringing all this together,
$$P(\omega_1,\dots,\omega_k)_{x=0}\,\overset{2.1}{=}\,\tilde P\left((\omega_1)_0,(\nabla\omega_1)_0,\dots,(\omega_k)_0,(\nabla\omega_k)_0\right)
$$
$$=\,\sum_{\{ d_{i,j}\}}\tilde P^{\{ d_{i,j}\}}\left(\left[(\omega_1)_0\otimes\overset{d_{1,0}}{\dots}\otimes(\omega_1)_0\right]\otimes\left[(\nabla\omega_1)_0\otimes\overset{d_{1,1}}{\dots}\otimes(\nabla\omega_1)_0\right]\otimes\dots\right.$$
$$ \left. \dots\otimes
\left[(\omega_k)_0\otimes\overset{d_{k,0}}{\dots}\otimes(\omega_k)_0\right]\otimes\left[(\nabla\omega_k)_0\otimes\overset{d_{k,1}}{\dots}\otimes(\nabla\omega_1)_0\right]\right)$$
$$=\,\sum_{\{ d_{i,j}\}}\lambda^{\{ d_{i,j}\}}h\left(\left[(\omega_1)_0\otimes\overset{d_{1,0}}{\dots}\otimes(\omega_1)_0\right]\otimes\left[(\nabla\omega_1)_0\otimes\overset{d_{1,1}}{\dots}\otimes(\nabla\omega_1)_0\right]\otimes\dots\right.$$
$$ \left. \dots\otimes
\left[(\omega_k)_0\otimes\overset{d_{k,0}}{\dots}\otimes(\omega_k)_0\right]\otimes\left[(\nabla\omega_k)_0\otimes\overset{d_{k,1}}{\dots}\otimes(\nabla\omega_1)_0\right]\right)$$
(using Properties \ref{0.4} of the skew-symmetrisation operator)
$$=\,\sum_{\{ d_{i,j}\}}\mu^{\{ d_{i,j}\}}\left(\left[(\omega_1)_0\wedge\overset{d_{1,0}}{\dots}\wedge(\omega_1)_0\right]\wedge\left[(\D\omega_1)_0\wedge\overset{d_{1,1}}{\dots}\wedge(\D\omega_1)_0\right]\wedge\dots\right.$$
$$ \left. \dots\wedge
\left[(\omega_k)_0\wedge\overset{d_{k,0}}{\dots}\wedge(\omega_k)_0\right]\wedge\left[(\D\omega_k)_0\wedge\overset{d_{k,1}}{\dots}\wedge(\D\omega_1)_0\right]\right)$$
where $\mu^{\{ d_{i,j}\}}:=\lambda^{\{ d_{i,j}\}}(p_1!)^{d_{1,0}+d_{1,1}}\cdots(p_k!)^{d_{k,0}+d_{k,1}}$.

Therefore, taking
$$\mathbf{P}(u_1,v_1,\dots,u_k,v_k):=\,\sum_{\{ d_{i,j}\}}\mu^{\{ d_{i,j}\}}u_1^{d_{1,0}}v_1^{d_{1,1}}\cdots u_k^{d_{k,0}}v_k^{d_{k,1}}$$
it follows 
$$P(\omega_1,\dots,\omega_r)_{x=0}\,=\, \mathbf{P}(\omega_1,\mathrm{d}\omega_1,\dots,\omega_k,\mathrm{d}\omega_k)_{x=0}\ .$$

By naturalness, we conclude
\begin{equation}\label{CondUnicidad}
P(\omega_1,\dots,\omega_k)\,=\, \mathbf{P}(\omega_1,\mathrm{d}\omega_1,\dots,\omega_k,\mathrm{d}\omega_k) 
\end{equation}
for any $n$-manifold $\,X\,$ and any $\,(\omega_1,\dots,\omega_k)\in \Omega^{p_1}(X)\oplus\dots\oplus\Omega^{p_k}(X)$.

Finally, the uniqueness of the polynomial $\,\mathbf{P}\,$ is proved applying Lemma \ref{2.4} below to the difference of two polynomials satisfying (\ref{CondUnicidad}).

\qed

\begin{lema}\label{2.4}
Let $\,\mathbb{R}\{ u_1,v_1,\dots,u_k,v_k\}\,$ be as in Theorem \ref{2.3}.

Let $\,\mathbf{P}\in\mathbb{R}\{ u_1,v_1,\dots,u_k,v_k\}\,$ be a non-zero homogeneous polynomial of degree $\,q\leq n$. Then, the morphism of sheaves over $\,X=\mathbb{R}^n$
$$ P\colon \Omega^{p_1}_X \oplus \ldots \oplus \Omega^{p_k}_X \  \longrightarrow \ \Omega^q_X\quad ,\qquad P(\omega_1,\dots,\omega_k)=\mathbf{P}(\omega_1,\D\omega_1,\dots,\omega_k,\D\omega_k)  $$
is not identically zero.
\end{lema}

\proof Let us first consider the case $\,k=1$, that is, $\mathbf{P}\in\mathbb{R}\{u,v\}$.

In this case, $\,\mathbf{P}\,$ is (up to a scalar factor) a monomial of one of the following four types, depending on the parities of $\,p_1\,$ and $\,q$,
$$\left\{\begin{aligned}
&\ u^0v^s\qquad (p_1\text{ odd, }q \text{ even}) \\
&\ u^1v^s\qquad (p_1\text{ odd, }q\text{ odd}) \\
&\ u^sv^0\qquad (p_1\text{ even, }q \text{ even}) \\
&\ u^sv^1\qquad (p_1\text{ even, }q\text{ odd})\end{aligned}\right.$$
where $\,s\,$ has to be taken in each case, for the monomial to have degree $\,q$.

Depending on each of the four cases, let us consider the $\,p_1$-form
$$\omega_1\,=\left\{\begin{aligned}
\sum_{j=1}^sy_{j_0}\D y_{j_1}\wedge\dots\wedge\D y_{j_{p_1}} \quad , \quad
\text{case }u^0v^s\\
\D x_1\wedge\dots\wedge\D x_{p_1}+\sum_{j=1}^sy_{j_0}\D y_{j_1}\wedge\dots\wedge\D y_{j_{p_1}}\quad , \quad \text{case }u^1v^s\\
\sum_{j=1}^s\D y_{j_1}\wedge\dots\wedge\D y_{j_{p_1}} \quad , \quad
\text{case }u^sv^0\\
x_0\D x_1\wedge\dots\wedge\D x_{p_1}+\sum_{j=1}^s\D y_{j_1}\wedge\dots\wedge\D y_{j_{p_1}} \quad , \quad \text{case }u^sv^1
\end{aligned}\right.$$
on $\,X=\mathbb{R}^q$ (each formula employs $\,q\,$ variables). In each of the four cases, it holds
$$P(\omega_1)\,=\,\mathbf{P}(\omega_1,\D\omega_1)\,=\left\{\begin{aligned} 
(\omega_1)^0\wedge(\D\omega_1)^s\,=\, s! \,\omega_{\mathbb{R}^q} \\
(\omega_1)^1\wedge(\D\omega_1)^s\,=\, s! \,\omega_{\mathbb{R}^q} \\
(\omega_1)^s\wedge(\D\omega_1)^0\,=\, s! \,\omega_{\mathbb{R}^q} \\
(\omega_1)^s\wedge(\D\omega_1)^1\,=\, s! \,\omega_{\mathbb{R}^q}
\end{aligned}\right.$$
where $\,\omega_{\mathbb{R}^q}\,$ is the volume form of $\,\mathbb{R}^q$. Hence, the morphism $\,P\colon \Omega^{p_1}_X\longrightarrow \Omega^q_X\,$ is not identically zero for $\,X=\mathbb{R}^q$.

Let us now consider the case where $\,k\,$ is arbitrary. Let
$$\mathbf{P}\,=\,\sum\lambda_{a_1b_1\dots a_kb_k}\, u_1^{a_1}v_1^{b_1}\dots u_k^{a_k}v_k^{b_k}$$ be a homogeneous polynomial of degree $\,q$. Let us fix an index $\,a_1b_1\dots a_kb_k\,$ whose corresponding coefficient is non-zero. On each pair $\,a_ib_i\,$ one of the two terms is $0$ or $1$, so that we have again four possibilities $\,0s_i$, $\,1s_i$, $\,s_i0$ and $\,s_i1$. Let us now define on $\,\mathbb{R}^q\,$ a $p_i$-form $\,\omega_i\,$ using the four formulas above, depending on the case (but 
taking care of writing the different forms $\,\omega_1,\dots,\omega_k\,$ on $\,\mathbb{R}^q$ with disjoint variables).

It is now easy to check that:
$$P(\omega_1,\dots,\omega_k)\,=\,\mathbf{P}(\omega_1,\D\omega_1,\dots,\omega_k,\D\omega_k)\,=\,\lambda_{a_1b_1\dots a_kb_k}s_1!\cdots s_k!\cdot\omega_{\mathbb{R}^q}\ ,$$
so that $\,P\colon \Omega^{p_1}_X \oplus \ldots \oplus \Omega^{p_k}_X\longrightarrow \Omega^q_X\,$ is not identically zero when $\,X=\mathbb{R}^q$.
 
Finally, the same expressions for $\,\omega_1,\dots,\omega_k\,$ prove the thesis for $\,X=\mathbb{R}^n\,$ when $\,n\geq q$.

\qed

\medskip
A particular case of Theorem \ref{2.3} is the following characterisation of the exterior differential, that was first proved for $\mathbb{R}$-linear morphisms by Palais (\cite{Palais}, Thm. 10.5), and in the general case by Kol\'{a}\v{r}-Michor-Slov\'{a}k (\cite{KMS}, Prop. 25.4).

\begin{corolario} Let $\,p\,$ be a positive integer. Up to a constant factor, the only regular and natural morphism  $\,P\colon \Omega^p\to\Omega^{p+1}\,$ over $\,\mathbf{Man}_n\,$ is the exterior differential: $\,P(\omega)=\D\omega$.
\end{corolario}

\proof According to the previous Theorem, we have to consider the algebra $\,\mathbb{R}\{ u,v\}$, where $\,\text{deg}\, u=p$, $\,\text{deg}\, v=p+1$. In this algebra, the only (up to scalar factors) homogeneous polynomial of degree $\,p+1\,$ is the monomial $\,v$, that corresponds to the natural morphism $\,P\colon \Omega^p\to\Omega^{p+1}$, $\,P(\omega)=\D\omega$.

\qed

\section{Main result: second version}

Let us now prove a variation of Theorem \ref{2.3}, that corresponds to the statement announced in the Introduction.\medskip

\begin{notation}
Let \textbf{Man} be the category of all smooth manifolds and arbitrary smooth maps between them.
\end{notation}

Firstly, let us check that any morphism of functors over $\,\textbf{Man}\,$ automatically satisfies the regularity condition.

\begin{lema} Let $\,P\colon\Omega^p\to\Omega^q\,$ be a morphism of functors over the category $\,{\bf Man}$. The restriction of $\,P\,$ to a smooth manifold $\,X\,$ is a regular morphism $\,P\colon\Omega^p_X\to\Omega_X^q\,$ of sheaves over $X$.
\end{lema}

\proof Let $\,T\,$ be a smooth manifold (of parameters), and let us consider an open set $\,U\subseteq X\times T\,$ and a smooth family $\,\{\omega_t\in\Omega^p_X(U_t)\}_{t\in T}\,$ of $p$-forms on the open sets $\,U_t=U\cap(X\times t)$.
Regularity is a local condition, so we can assume $\,X=\mathbb{R}^n$, $\,T=\mathbb{R}^k\,$ and $\,U=X\times T=\mathbb{R}^n\times\mathbb{R}^k$. Let us write 
$$\omega_t\,=\, \sum_{i_1<\dots<i_p} f_{i_1\dots i_p}(x,t)\, \text{d}x_{i_1}\wedge\dots\wedge\text{d}x_{i_p}\,\in\,\Omega_X^p(U_t=\mathbb{R}^n)\ .$$ 

The smoothness of the family $\,\{\omega_t\in\Omega^p_X(U_t)\}_{t\in T}\,$ means that the functions $\,f_{i_1\dots i_p}(x,t)\,$ are smooth on $\,\mathbb{R}^n\times\mathbb{R}^k$, that is,
$$\omega:=\, \sum_{i_1<\dots<i_p} f_{i_1\dots i_p}(x,t)\,\text{d}x_{i_1}\wedge\dots\wedge\text{d}x_{i_p}$$
is a differential $p$-form  on $\,\mathbb{R}^n\times\mathbb{R}^k$. Then, $\,P(\omega)\in\Omega^q(\mathbb{R}^n\times\mathbb{R}^k)$, let us say
$$P(\omega)\,=\, \sum_{j_1<\dots<j_q} g_{j_1\dots j_q}(x,t)\,\text{d}x_{j_1}\wedge\dots\wedge\text{d}x_{j_q}+\cdots\,\text{ terms with  }\text{d}t_j\,\cdots$$
for some smooth functions $\,g_{j_1\dots j_q}(x,t)\,$ on $\,\mathbb{R}^n\times\mathbb{R}^k$.

Now we can write
$$P(\omega_t)\,=\, P(\omega_{|\mathbb{R}^n\times t})$$
(by functoriality)
$$=\, P(\omega)_{|\mathbb{R}^n\times t}\,=\, \sum_{j_1<\dots<j_q} g_{j_1\dots j_q}(x,t)\,\text{d}x_{j_1}\wedge\dots\wedge\text{d}x_{j_q}\ ,$$ 
so that the family $\,\{ P(\omega_t)\}_{t\in T}\,$ is smooth.

\qed

\medskip
In a similar way, it can be proved that, if $\,P\colon\Omega^{p_1}\oplus\cdots\oplus\Omega^{p_k}\to\Omega^q\,$ is a morphism of functors over $\,\textbf{Man}$, then its restriction to a smooth manifold $\,X\,$ is a regular morphism $\, P\colon\Omega^{p_1}_X\oplus\cdots\oplus\Omega^{p_k}_X\to\Omega^q_X\,$ of sheaves over $X$. 
Consequently,

\begin{corolario} Let $\,P\colon\Omega^{p_1}\oplus\cdots\oplus\Omega^{p_k}\to\Omega^q\,$ be a morphism of functors over $\,\mathbf{Man}$. Its restriction to the category $\,\mathbf{Man}_n\,$ is a regular and natural morphism.
\end{corolario}

\begin{teorema}\label{3.3} Let $\,p_1,\dots,p_k\,$ be positive integers, and let $\,\mathbb{R}\{ u_1,v_1,\dots,u_k,v_k\}\,$ be the anti-commutative algebra of polynomials in the variables $\,u_1,v_1,\dots,u_k,v_k$, of degree $\,\mathrm{deg}\, u_i=p_i$, $\,\mathrm{deg}\, v_i=p_i+1$.

Any morphism of functors over the category $\,\mathbf{Man}$
$$ P\colon \Omega^{p_1} \oplus \ldots \oplus \Omega^{p_k} \  \longrightarrow \ \Omega^q\qquad (q\geq 0)  $$
can be written as
$$P(\omega_1,\dots,\omega_k)\,=\, \mathbf{P}(\omega_1,\mathrm{d}\omega_1,\dots,\omega_k,\mathrm{d}\omega_k)$$
for a unique homogeneous polynomial $\,\mathbf{P}(u_1,v_1,\dots,u_k,v_k)\in\mathbb{R}\{ u_1,v_1,\dots,u_k,v_k\}$ of degree $\, q$.
\end{teorema}

\proof Let $\,P_n\,$ be the restriction of $\,P\,$ to the category $\,\mathbf{Man}_n$. By Theorem \ref{2.3}, for $\,n\geq q$ there exists a unique homogeneous polynomial $\,\mathbf{P}_n\in\mathbb{R}\{ u_1,v_1,\dots,u_k,v_k\}\,$ of degree $\,q\,$ such that 
$$ P(\omega_1,\dots,\omega_k)\,=\,P_n(\omega_1,\dots,\omega_k)\,=\, \mathbf{P}_n(\omega_1,\mathrm{d}\omega_1,\dots,\omega_k,\mathrm{d}\omega_k)$$
for any $n$-manifold $\,X\,$ and any $\,(\omega_1,\dots,\omega_k)\in \Omega^{p_1}(X)\oplus\dots\oplus\Omega^{p_k}(X)$.

Let us check that $\,\mathbf{P}_n=\mathbf{P}_m\,$ for all $\,n\,$ and $\,m\,$ greater that $\,q$. Suppose $\,m<n$. For any $m$-manifold $\,X_m$, consider the $n$-manifold  $\,X_m\times\mathbb{R}^{n-m}$, the natural projection $\,\pi\colon X_m\times\mathbb{R}^{n-m}\to X_m\,$ and the inclusion $\,i\colon X_m\hookrightarrow X_m\times\mathbb{R}^{n-m}$, $\,x\mapsto(x,0)$, so that $\,\pi\circ i=\text{Id}$. 

For any $\,(\omega_1,\dots,\omega_k)\in \Omega^{p_1}(X_m)\oplus\dots\oplus\Omega^{p_k}(X_m)$, on the one hand
$$P_m(\omega_1,\dots,\omega_k)\,=\, \mathbf{P}_m(\omega_1,\mathrm{d}\omega_1,\dots,\omega_k,\mathrm{d}\omega_k) \ , $$
and, on the other,
\begin{align*}
P_m(&\omega_1,\dots,\omega_k)\,=\, P(\omega_1,\dots,\omega_k)\,=\,P(i^*\pi^*\omega_1,\dots,i^*\pi^*\omega_k)\\
&\,=\,i^*P(\pi^*\omega_1,\dots,\pi^*\omega_k)\,=\,i^*\mathbf{P}_n(\pi^*\omega_1,\pi^*\mathrm{d}\omega_1,\dots,\pi^*\omega_k,\pi^*\mathrm{d}\omega_k)\\
&\,=\,\mathbf{P}_n(i^*\pi^*\omega_1,i^*\pi^*\mathrm{d}\omega_1,\dots,i^*\pi^*\omega_k,i^*\pi^*\mathrm{d}\omega_k)\,=\,\mathbf{P}_n(\omega_1,\mathrm{d}\omega_1,\dots,\omega_k,\mathrm{d}\omega_k)
\end{align*}
so that $\,\mathbf{P}_n=\mathbf{P}_m$.

Let us write $\,\mathbf{P}\,$ for this identical polynomials $\,\mathbf{P}_n$. By definition of the $\,\mathbf{P}_n$, it holds
$$P(\omega_1,\dots,\omega_k)\,=\, \mathbf{P}(\omega_1,\mathrm{d}\omega_1,\dots,\omega_k,\mathrm{d}\omega_k)$$
for any $\,(\omega_1,\dots,\omega_k)\in \Omega^{p_1}(X)\oplus\dots\oplus\Omega^{p_k}(X)\,$ on any manifold $\,X\,$ of dimension $\,n\geq q$. This equality also holds for manifolds $\,X\,$ of dimension $\,n<q$, because in such a case, both terms of the formula are zero, for they take values on $\,\Omega^q(X)=0$.

\qed

The case $\,p_1=\dots=p_k=1\,$ of this theorem has been proved by Freed-Hopkins (\cite{FreedHopkins}, Thm. 8.1).

\section{Appendix: $\g$-Valued forms associated to a connection}
 
Let us fix a Lie group $G$ and let $\mathfrak{g}$ be its Lie algebra. Given a principal $G$-bundle $P\to X$ with a principal connection $\alpha$ on it, let us denote $\Theta=\Theta(P,\alpha)$ the curvature form, which is a $\g$-valued $2$-form on $P$ with the following standard properties:
\begin{enumerate}
\item[{\it 1.}] {\it It is horizontal}: If $D$ is a vector field
tangent to the fibres of $P \to X$, then
$$i_D \Theta = 0 \ . $$

\item[{\it 2.}] {\it It is $G$-equivariant}: For any element $g \in G$, it holds:
$$ (R_g^* \Theta) (D , \bar{D} ) \, = \, \Ad (g^{-1}) (\Theta (D , \bar{D} )) \ ,  $$
where $R_g \colon P \to P$, $R_g (p) = p \cdot g$, denotes right
translation by $g \in G$.

\item[{\it 3.}] {\it It is invariant under isomorphisms}: if $\Phi \colon (P , \alpha) \to (\bar{P} , \bar{\alpha})$ is an isomorphism of principal $G$-bundles over $X$ such that $\Phi^* \bar{\alpha} = \alpha$,
then
$$ \Phi^* \left[ \Theta (\bar{P} , \bar{\alpha}) \right] \, = \, \Theta (P , \alpha) \ . $$

\item[{\it 4.}] {\it It is stable under arbitrary base changes}: For any
smooth map $f \colon Y \to X$, it holds:
$$ f^* \left[ \Theta (P, \alpha) \right] \, = \, \Theta (f^* P , f^* \alpha )\ . $$
\end{enumerate}

In this section we shall show that the properties {\it 1-4} essentially characterise the curvature $2$-form. More generally, we will determine all the differential forms of arbitrary order satisfying the properties {\it 1-4}, with values on a linear representation of $G$ (Theorem \ref{TeoremaGeneral}).\medskip

Let $P\to X$ be a principal fibre bundle with group $G$. We follow the conventions of \cite{Kobayashi1}; in particular, $G$ acts on $P$ on the right.


\begin{lemma}\label{LemaUno}
Let $P= \mathbb{R}^n \times G \to \mathbb{R}^n$ be the trivial $G$-bundle, let $\alpha$ be a principal connection on it, let $s \colon \mathbb{R}^n \to P$ be a global section, and let $x_0 \in \mathbb{R}^n$ be a point.

There  exists an isomorphism $\Phi \colon P \to P$ such that:
$$ (s^* \Phi^* \alpha)_{x_0} \, = \, 0 \ . $$
\end{lemma}

\proof Let $\bar s\colon \mathbb{R}^n \to P$ be a global section such that $\bar s(x_0)=s(x_0)$ and $\bar s_{*}(T_{x_0}\mathbb{R}^n)=\text{ker}\,\alpha_{s(x_0)}$, i.e., $(\bar s^*\alpha)_{x_0}=0$. Let us write $s(x)=(x,f(x))$ and $\bar s(x)=(x,\bar f(x))$. Let $g(x)\in G$ be such that $\bar f(x)=g(x)\cdot f(x)$; then the isomorphism $\Phi\colon P\to P$, $\Phi(x,g)=(x,g(x)\cdot g)$, satisfies $\bar s=\Phi\circ s$, and hence $\,(s^* \Phi^* \alpha)_{x_0}=(\bar s^*\alpha)_{x_0}=0$.

\hfill$\square$ \medskip

 The adjoint action on the Lie algebra of an element $g \in G$ will be written $\Ad (g) \colon \g \to \g$. This map induces a linear map $\Ad (g) \colon S^q\g \to S^q\g$, where $S^q\g$ denotes the $q$-th symmetric  tensor power of $\g$; in other words, the adjoint action of $G$ over $\g$ induces an action on $S^q\g$.\medskip
 
The following statement is a reformulation of a result due to Kol\'{a}\v{r} (\cite{Kolar}). The original proof is based on a generalisation of the classical Utiyama theorem on gauge-invariant Lagrangians.

\begin{theorem}\label{TeoremaGeneral}
Let $G$ be a Lie group, let $V$ be a linear representation of $G$ and let $q$ be a natural number. 

Associated to any pair $(P \to X , \alpha)$ of a principal $G$-bundle and a principal connection $\alpha$ on it, let $\theta (P , \alpha)$ be a $2q$-form on $P$ with values on $V$, satisfying the properties {\it 1-4 } written above.

Then, there exists a $G$-equivariant linear map $T \colon S^q\g \to V$ such that
\begin{equation}\label{Tesis}
\theta (P , \alpha) \, = \, T \circ \left[ \Theta (P, \alpha) \, \wedge
\stackrel{q}{\ldots}
 \wedge \, \Theta (P , \alpha) \right] \ ,
\end{equation}  
for all pairs $(P \to X  , \, \alpha)$.

Any differential form of odd order $\theta (P, \alpha)$ satisfying
properties {\it 1-4} above is identically zero.
\end{theorem}

\begin{proof} Let $\theta (P, \alpha)$ be a $2q$-form as in the statement, and let us prove formula (\ref{Tesis}).

As any principal bundle is locally isomorphic to the trivial bundle, properties {\it 3} and {\it 4} allow to reduce the reasoning to the case of the trivial bundle $P = X \times G \to X$.

Moreover, due to properties {\it 1} and {\it 2}, the form $\theta (P, \alpha)$, for any connection $\alpha$ on the trivial bundle $P$,
is determined by its restriction to a fixed global section $s \colon X \to P$. 

Hence, it is enough to prove that there exists a $G$-equivariant linear map $T \colon S^q\g \to V$ such that 
\begin{equation}\label{ToProve}
\theta (P, \alpha)_{|s} = T \circ \left[ \Theta (P, \alpha)_{|s}\, \wedge \stackrel{q}{\ldots} \wedge \, \Theta (P , \alpha)_{|s} \right]  
\end{equation}
for any connection $\alpha$ on the trivial bundle $P$.

To this end, observe that property {\it 4} implies that the map:
\begin{equation}\label{MorfismoDeHaces}
\Omega^1 \otimes \g \longrightarrow
\Omega^{2q} \otimes V \quad , \quad \alpha_{|s} \longmapsto \theta (P,
\alpha)_{|s} \ ,
\end{equation} 
is a well-defined morphism of functors over the category $\Man$.

Let us fix a basis $(v_1, \ldots , v_s)$ of $V$, and write:
$$ \alpha = \sum_{j=1}^s \alpha_j \otimes v_j \quad , \quad \Theta ( P , \alpha) = \sum_{j=1}^s
\Theta_j (P , \alpha) \otimes v_j \quad , \quad \theta (P , \alpha) = \sum_{j=1}^s
\theta_j (P , \alpha) \otimes v_j $$ for some ordinary forms $\alpha_j , \Theta_j (P,
\alpha)$, $\theta_j (P, \alpha)$ on $P$.

Applying Theorem \ref{3.3} to the components of the morphism (\ref{MorfismoDeHaces}), it
follows the existence of a unique collection of scalars $\lambda^I_{ j} , \delta^{KL}_j
\in \RR$ such that, for any connection $\alpha$ on the trivial bundle $P$,
\begin{align*}
\theta_j (P , \alpha)_{|s} \  &=  \  \sum_{|I|= q} \lambda^{I}_{j} \, \d
(\alpha_{{i_1}|s} ) \, \wedge \stackrel{q}{\ldots} \wedge \, \d (\alpha_{{i_q}|s} )  \\
&  \quad + \, \sum_{ 2r + s  = 2q } \delta^{KL}_j \d
(\alpha_{{k_1}|s} )\, \wedge \stackrel{r}{\ldots} \wedge \, \d (\alpha_{{k_r}|s} ) \wedge (\alpha_{l_1})_{|s} 
\wedge \stackrel{s}{\ldots} \wedge (\alpha_{l_{s}})_{|s} \ .
\end{align*}

As $\Theta = \d \alpha +\frac{1}{2}\, \alpha \wedge \alpha$, the formula above can be re-written, 
with other unique\smallskip

\noindent coefficients $\lambda^I_{ j} , \mu^{KL}_j
\in \RR$,
\begin{align*}
\theta_j (P , \alpha)_{|s} \ &= \ \sum_{|I|=q} \lambda^I_{j} \ \Theta_{i_1} (P,
\alpha)_{|s}\,  \wedge \stackrel{q}{\ldots} \wedge \, \Theta_{i_q} (P , \alpha)_{|s} \\
& \quad + \, \sum_{ 2r + s  = 2q } \mu^{KL}_j \d
(\alpha_{{k_1}|s} )\, \wedge \stackrel{r}{\ldots} \wedge \,  \d (\alpha_{{k_r}|s} ) \, \wedge \, (\alpha_{l_1})_{|s} \,
\wedge \stackrel{s}{\ldots} \wedge \, (\alpha_{l_{s}})_{|s} \ .
\end{align*}

At any point $x \in X$, there always exists a connection $\bar{\alpha}=\Phi^*\alpha$, isomorphic to
$\alpha$, such that $(\bar{\alpha} _{|s})_{x}= 0$ (Lemma \ref{LemaUno}). As $\theta_j(P, \alpha)=(\Phi^{-1})^*[\theta_j
(P, \bar\alpha)]$, the coefficients $\mu_j^{KL}$
above have to be zero; that is to say,
$$ \theta_j (P , \alpha)_{|s} \, =\,  \sum_{|I|=q} \lambda^I_{j} \ \Theta_{i_1} (P,
\alpha)_{|s} \, \wedge \stackrel{q}{\ldots} \wedge \, \Theta_{i_q} (P , \alpha)_{|s}  \ , $$
which is the required formula (\ref{ToProve}), where 
 $T\colon S^q\g \to V$ is the linear map defined by the scalars $\lambda^I_{j}$.

To check the $G$-equivariance of $T$, let us apply $R_g^*$ in formula (\ref{Tesis}). As $\theta$ and $\Theta$ are $G$-equivariant, it follows:
$$ \Ad (g^{-1}) ( T( \Theta ( D_1 , \bar{D}_1) \wedge \ldots \wedge \Theta (D_q , \bar{D}_q) )$$
$$= \, T \left(  \Ad (g^{-1}) (\Theta (D_1 , \bar{D}_1) \wedge \ldots \wedge \Theta (D_q , \bar{D}_q) ) \right) $$
for any vector fields $D_1 ,  \ldots , \bar{D}_q$ on any principal bundle $P \to X$.

When varying the pair $(P \to X, \alpha)$, the values $\Theta (D_1 , \bar{D}_1) \wedge \ldots \wedge \Theta (D_q , \bar{D}_q) $ linearly span $S^q \g$ (see Lemma \ref{LemaDos} below). Hence, the relation above allows to conclude that $T \colon S^q \g \to V$ is $G$-equivariant.

Finally, if $\theta (P, \alpha)$ is a $(2q-1)$-form, a similar reasoning applies. Nevertheless, in this case Theorem \ref{3.3} implies that (\ref{MorfismoDeHaces}) is the zero map, and therefore $\theta (P, \alpha)$ is the zero form.

\qed
\end{proof}

\medskip
Observe that the last assertion of this theorem implies that the $2q$-forms $\,\theta=\theta(P,\alpha)\,$ have vanishing covariant differential; that is, they satisfy a kind of Bianchi identity.

\begin{lemma}\label{LemaDos}
Let $\,v_1\cdots v_q\in S^q\g\,$  be a product of elements $\, v_1,\dots,v_q\in\mathfrak{g}$.
There exists a principal bundle $P \to X$, a connection $\alpha $ on it, a point $p \in P$ and vectors $D_1, \ldots , D_{2q} \in T_pP$ such that:
$$ \left( \Theta \, \wedge \stackrel{q}{\cdots} \wedge \, \Theta \right)_p (D_1, \ldots , D_{2q} ) \, = \, v_1\cdots v_q \ , $$
where $\Theta$ denotes the curvature form of $\alpha$.
\end{lemma}

\proof Let $P =\mathbb{R}^{2q}\times G \to \mathbb{R}^{2q}$ be the trivial bundle and let $(x_1, \ldots ,x_q,y_1,\ldots , y_q)$ be linear coordinates on $\mathbb{R}^{2q}$. Let $s\colon\mathbb{R}^{2q}\to P$ be the unit section: $s(x)=(x,1)$.

Let us consider the unique connection $\alpha$ on $P$ such that
$$\alpha_{|s}\,=\, \sum_{i=1}^qx_i\d y_i\otimes v_i\ .$$

The restriction to $s$ of the curvature form $\Theta=\d\alpha+\frac{1}{2}\alpha\wedge\alpha$ is
$$\Theta_{|s}\,=\, \sum_i \d x_i\wedge\d y_i\otimes v_i+\frac{1}{2}\sum_{i,j}x_ix_j\d y_i\wedge\d y_j\otimes[v_i,v_j]\ ,$$
hence
$$\left( \Theta \, \wedge \stackrel{q}{\cdots} \wedge \, \Omega \right)_{|s}\,=\, q!\,(\d x_1\wedge\d y_1\wedge\cdots\wedge\d x_q\wedge\d y_q)\otimes(v_1\cdots v_q)\ .$$

Therefore, at any point $p$ of the unit section $s$, we have
$$\left( \Theta \, \wedge \stackrel{q}{\cdots} \wedge \, \Theta \right)_p (\partial_{x_1},\partial_{y_1},\dots,\partial_{x_q},\partial_{y_q})\,=\, q!\, v_1\cdots v_q\ .$$

\hfill$\square$\medskip

 The case $q=1$, $V=\g$ in Theorem \ref{TeoremaGeneral} is essentially a characterisation of the curvature form:
 
\begin{corollary}\label{caso q=1}
Associated to any pair $(P \to X , \alpha)$ of a principal $G$-bundle and a principal connection $\alpha$ on it, let $\theta (P , \alpha)$ be a $\g$-valued $2$-form on $P$ 
satisfying properties {\it 1-4} above.

Then, there exists a $G$-equivariant endomorphism $T \colon \g \to \g$ such that\begin{equation}
\theta (P , \alpha) \, = \, T \circ  \Theta (P, \alpha) \ ,
\end{equation}
for all pairs $\,(P \to X  , \, \alpha)$.
\end{corollary}

If the Lie group $G$ is simple, then the adjoint representation  is irreducible. This is the case of the compact groups $U(1)$, $SU(n)$, $SO(n\neq 4)$, $Sp(n)$, $Spin(n\neq 4)$, and the exceptional groups $G_2$, $F_4$, $E_6$, $E_7$ and $E_8$.
Then any $G$-equivariant endomorphism $T \colon \g \to \g$ is an homothety, and the previous corollary reads: 

\begin{corollary}\label{Unitario} Let $G$ be a simple Lie group.
Associated to any pair $(P \to X , \alpha)$ of a principal $G$-bundle and a principal connection $\alpha$ on it, let $\theta (P , \alpha)$ be a $\mathfrak{g}$-valued $2$-form on $P$ satisfying the properties {\it 1-4} above.

Then, there exists $\lambda \in \mathbb{R}$ such that, for all pairs $(P \to X  , \, \alpha)$,
$$ \theta (P , \alpha) \, = \, \lambda \cdot \Theta (P , \alpha) \ . $$ 
\end{corollary}
  
\medskip
Let us now consider the trivial representation $V=\mathbb{R}$ in Theorem \ref{TeoremaGeneral}. Now $\theta(P,\alpha)$ is an ordinary form on the principal bundle $\pi\colon P\to X$ satisfying the properties {\it 1}-{\it 4}. From properties {\it 1} and {\it 2}, it follows that $\theta(P,\alpha)=\pi^*\widetilde\theta(P,\alpha)$ for an unique ordinary form $\widetilde\theta(P,\alpha)$ on the base manifold $X$. Therefore we may reformulate Theorem \ref{TeoremaGeneral} (in the case $V=\mathbb{R}$) using the language of ordinary forms on base manifolds.  Let us precise the statement:

Let us fix a Lie group $G$. Let $\mathcal{C}$ denote the  functor on $\Man$ that assigns, to any smooth manifold $X$, the set of isomorphy classes of pairs $(P \to X , \alpha)$, where $P \to X$ is a principal $G$-bundle and $\alpha$ is a principal connection on it.

A \textbf{ $q$-form naturally associated to a connection} is a morphism of functors $\widetilde\theta \colon \mathcal{C} \to \Omega^q$.

By definition, $\,\widetilde\theta\,$ assigns, to any pair $\,(P \to X, \alpha)\,$, a differential $q$-form $\widetilde\theta (P, \alpha)$ on $\,X\,$ satisfying the following two properties:
\begin{itemize}
\item[-] If $\,(P \to X, \alpha)\,$ and $\,(P' \to X , \alpha')\,$ are isomorphic, then $\,\widetilde\theta (P, \alpha) = \widetilde\theta(P',\alpha')$.

\item[-] For any pair $\,(P \to X , \alpha)\,$ and any smooth map $\,f \colon Y \to X$, it holds:
$$ \widetilde\theta ( f^* P , f^* \alpha) = f^* (\widetilde\theta (P, \alpha))\ . $$
\end{itemize}

The following statement is again a reformulation of a theorem due to Kol\'{a}\v{r} (\cite{Kolar}). See also (\cite{FreedHopkins}, Th. 7.20) for another formulation.

\begin{corollary}  The $2q$-forms naturally associated to a connection  biunivocally correspond with the $G$-invariant linear maps $T \colon S^q \g \to \RR$.

Any form of odd order naturally associated to a connection is identically zero.
\end{corollary}

More precisely, the natural $2q$-form $\,\widetilde\theta(P,\alpha)\,$ corresponding to a $G$-invariant linear map $\,T \colon S^q \g \to \RR\,$ is the projection on $\,X\,$ of the following $2q$-form over $\,P\,$ 
$$\theta (P, \alpha) \, := \, T \circ \left[ \Theta (P, \alpha) \, \wedge \stackrel{q}{\ldots}  \wedge \, \Theta (P, \alpha) \right] \ , $$
where $\,\Theta\,$ is the $\mathfrak{g}$-valued curvature $2$-form on $\,P$. 

This result shows that the Chern-Weil forms, defined by the Weil homomorphism (\cite{Kobayashi2}), are the {\it only} natural differential forms one can construct from a $G$-connection. 

Observe that the second part of the corollary automatically implies that any $2q$-form naturally associated to a connection is closed. \medskip

\begin{example} If $G = Gl_n (\mathbb{C})$ is the general complex linear group of rank $n$, the algebra of $Gl_n$-invariant, real polynomials $T(x_{rs},y_{rs})$ on $\mathfrak{gl}_n=M_n (\CC)$ is generated by the  real and imaginary parts of the coefficients $c_1(z_{rs}), \ldots , c_n(z_{rs})$ of the characteristic polynomial of each matrix $(z_{rs}=x_{rs}+iy_{rs})$. 
The corresponding forms associated to a $G$-connection are essentially the Chern forms.
\end{example}


\begin{thebibliography}{00}


\bibitem{FreedHopkins} Freed, D. S.; Hopkins, M. J.: \emph{Chern-Weil forms and abstract homotopy theory}, Bull. Amer. Math. Soc., {\bf 50} (2013), 431--468

\bibitem{Kolar} Kol\'{a}\v{r}, I.: \emph{Gauge-natural forms of Chern-Weil type}, Ann. Global Anal. Geom. 
\textbf{11}, 41-47 (1993)

\bibitem{KMS} Kol\'{a}\v{r}, I., Michor, P.W. \& Slov\'{a}k, J.: \emph{Natural operations in differential geometry},
Springer-Verlag, Berlin (1993)



\bibitem{Kobayashi1} Kobayashi S., Nomizu K.: \emph{Foundations of differential geometry}, Vol I, Wiley-Interscience Publishers, New York - London,  (1963) 

\bibitem{Kobayashi2} Kobayashi S., Nomizu K.: \emph{Foundations of differential geometry}, Vol II, Wiley-Interscience Publishers, New York - London,  (1969) 


\bibitem{PeetreRevisited} Navarro, J.; Sancho, J. B.: \emph{Peetre-Slov\'{a}k theorem revisited}, ArXiv: 1411.7499  

\bibitem{Palais} Palais, R. S. \emph{Natural operations on differential forms}, Trans. Amer. Math. Soc. {\bf 92}, 125--141 (1959) 

\bibitem{Peetre} Slov\'{a}k, J. \emph{Peetre theorem for nonlinear operators}, Ann. Global Anal. Geom. 
\textbf{6}, 273-283 (1988)



\end{thebibliography}
\end{document}